\documentclass{amsart}
\usepackage{amsmath,amsfonts,bm}
\usepackage[]{hyperref}
\newcommand{\FF}{\mathbb{F}}
\newcommand{\TT}{\mathbb{T}}
\newcommand{\CC}{\mathbb{C}}
\newcommand{\ZZ}{\mathbb{Z}}
\newcommand{\QQ}{\mathbb{Q}}
\newcommand{\pfrak}{\mathfrak{p}}
\newcommand{\qfrak}{\mathfrak{q}}
\newcommand{\cfrak}{\mathfrak{c}}

\newcommand{\ffrak}{\mathfrak{f}}
\newcommand{\Gal}{\operatorname{Gal}}
\newcommand{\pitilde}{\widetilde{\pi}}

\newcommand{\GL}{\operatorname{GL}}

\theoremstyle{plain}
\newtheorem{thm}{Theorem}[section]
\newtheorem*{thm*}{Theorem}
\newtheorem{cor}[thm]{Corollary}
\newtheorem{lem}[thm]{Lemma}
\newtheorem{prop}[thm]{Proposition}

\theoremstyle{definition}

\newtheorem{rem}[thm]{Remark}

\title[An Integral Digit Derivative Basis]{An Integral Digit Derivative Basis for Carlitz Prime Power Torsion Extensions}
\author{A. Maurischat}
\address{Lehrstuhl A f\"ur Mathematik, RWTH Aachen University,
Germany}
\email{andreas.maurischat@matha.rwth-aachen.de}

\author{R. Perkins}
\date{\today}
\address{IWR, University of Heidelberg, Im Neuenheimer Feld 205, 69120 Heidelberg, Germany}
\email{rudolph.perkins@iwr.uni-heidelberg.de}
\thanks{The second named author gratefully acknowledges the support of The Alexander von Humboldt Foundation.}

\begin{document}

\begin{abstract}
Let $\pfrak$ be a monic irreducible polynomial in $A:=\FF_q[\theta]$, the ring of polynomials in the  indeterminate $\theta$ over the finite field $\FF_q$, and let $\zeta$ be a root of $\pfrak$ in an algebraic closure of $\FF_q(\theta)$. For each positive integer $n$, let $\lambda_n$ be a generator of the $A$-module of Carlitz $\pfrak^n$-torsion. We give a basis for the ring of integers $A[\zeta,\lambda_n] \subset K(\zeta, \lambda_n)$ over $A[\zeta] \subset K(\zeta)$ which consists of monomials in the hyperderivatives of the Anderson-Thakur function $\omega$ evaluated at the roots of $\pfrak$. We also give an explicit field normal basis for these extensions. This builds on (and in some places, simplifies) the work of Angl\`es-Pellarin.
\end{abstract}

\maketitle

\section{Introduction}

\subsection{Carlitz torsion extensions}
Let $\FF_q$ be the finite field with $q$ elements and $\theta$ an indeterminate over this field. The integral extensions of $A := \FF_q[\theta]$ which interest us arise via the {\it Carlitz module functor} which takes as input an $A$-algebra $R$ and returns the $A$-module consisting of the additive group of $R$ with $A$-action given by mapping $X \in R$ to the evaluation at $X$ of the twisted polynomials determined by $\FF_q$-algebra map
\[\theta \mapsto \cfrak_\theta := \tau + \theta \in A\{\tau\},\]
and where evaluation is defined by $\tau^i(X) := X^{q^i}$. For example, $\theta^2 \in A$ acts via
\[\cfrak_{\theta^2}(X) = \cfrak_\theta( \cfrak_\theta(X))= X^{q^2} + (\theta^q+\theta)X^q + \theta^2 X .\]

Let $K := \FF_q(\theta)$. To the Carlitz module one associates an $\FF_q$-linear entire power series
\[\exp_C(X) := \sum_{j \geq 0} \frac{X^{q^j}}{D_j} \in K[[X]],\]
called the {\it Carlitz exponential}, where $D_j$ is the product of all monic polynomials in $A$ of degree $j$. This is the unique $\FF_q$-linear entire series such that $\frac{d}{dX}\exp_C = 1$ and
\[ \exp_C(\theta X) = \cfrak_\theta(\exp_C(X)). \]
We write $\pitilde A$ for the kernel of $\exp_C$ which is a free rank one $A$-submodule of $\CC_\infty$, the completion of the algebraic closure of $K_\infty := \FF_q((1/\theta))$ equipped with the canonical extension $|\cdot|$ of the absolute value for which $K_\infty$ is complete and normalized so that $|\theta| = q$;
we will make an explicit choice of $\pitilde$ just below.

For all $a \in A$, the polynomial $\cfrak_a(X)$ is separable in $X$, and adjoining its roots --- the so called Carlitz-$a$-torsion $C[a]$ --- to $K$ gives rise to a {\it Carlitz torsion extension} $K(C[a])$.
These extensions are a function field analog of the classical cyclotomic extensions of the rational integers, and one may consult Rosen's book \cite[Chapter 11]{Rosen} for a thorough summary of their basic properties which we take for granted in this note.

We will build on the work of Angl\`es-Pellarin \cite{APinv} focusing on the {\it Carlitz prime power torsion extensions} of $K$,
\[ K_n:=K(C[\pfrak^{n+1}]); n \geq 0,\]
arising from the Carlitz $\pfrak^n$-torsion of a fixed monic irreducible $\pfrak$. The ring of integers of $K_n$ will be denoted by $A_n$.
In fact, we have $A_n=A[x_n]$, for $x_n$ an $A$-module generator of $C[\pfrak^{n+1}]$ (e.g.~$x_n=\exp_C(\pitilde/\pfrak^{n+1})$).
The Galois group $\Gal(K_n/K)$ is isomorphic to $(A/\pfrak^{n+1} A)^\times$
via
\begin{equation}\label{eq:galois-action}
 (A/\pfrak^{n+1} A)^\times \cong \Gal(K_n/K),a\to \sigma_a
\end{equation}
given by $\sigma_a(x)=\cfrak_a(x)$ for all $x\in C[\pfrak^{n+1}]$.
Hence, one has
\[ [K_n: K] = |(A/\pfrak^{n+1} A)^\times| = |\pfrak|^{n+1} - |\pfrak|^{n}.\]
Hence, the degree of the extension $K_n/K$ is highly divisible by the characteristic of $K$, which makes the Galois module structure of the rings of integers in these extensions more delicate than for the cyclotomic extensions of $\mathbb{Q}$, and dealing with this gives rise to many interesting new phenomena not present in the classical setting.

\subsection{Angl\`es-Pellarin}
One has the now classical analytic description of all Carlitz torsion points in terms of the division values of the Carlitz exponential
\[ \bigcup_{a \in A} C[a] = \{ \exp_C(\pitilde \kappa) : \kappa \in K\}, \]
thus fulfilling the dream of Kronecker's youth or the analog of Hilbert's 12th problem for these extensions.
Remarkably, in a recent discovery by Angl\`es and Pellarin \cite{APinv}, another separate analytic description has been given for generators of these same Carlitz torsion extensions, after allowing for a finite constant extension of $K$, i.e. one obtained by adjoining an element algebraic over $\FF_q$ to $K$. The meromorphic function appearing in the work of Angl\`es-Pellarin {\it ibid.} arises from the scattering matrix for the Carlitz module, the so-called {\it Anderson-Thakur function} $\omega$.

To describe this function, we let $t$ be another indeterminate over $\CC_\infty$, and we fix an element $\lambda_\theta \in \CC_\infty$ of Carlitz $\theta$-torsion, so satisfying $\lambda_\theta^{q-1} = -\theta$.
We let
\[ \omega(t) := \lambda_\theta \prod_{j \geq 0} (1 - \frac{t}{\theta^{q^j}})^{-1} \in \CC_\infty[[t]]. \]
One easily checks that $\omega$ converges for $|t| < |\theta|$ and generates the free $\FF_q[t]$-submodule of $\CC[[t]]$ consisting of those functions both satisfying \begin{equation} \label{omegadiffeq} \sum_{i \geq 0} c_i^q t^i = (t - \theta) \sum_{i \geq 0} c_i t^i
\end{equation}
and converging on $\{t \in \CC_\infty : |t| \leq 1 \}$.
Further, by Anderson's general theory of scattering matrices (see \cite[Proposition 3.3.2]{AGtmot} and \cite[\S 2.5]{AGTDann}), the residue of $\omega$ at $t = \theta$ gives rise to a fundamental period of the Carlitz module, which we fix from now on
\[ \pitilde := - \text{res}_{t = \theta} \ \omega = -\lim_{t \rightarrow \theta} (t - \theta) \omega(t). \]

Let $\zeta$ be a root of the monic irreducible polynomial $\pfrak$ fixed above in the algebraic closure of $\FF_q$ inside $\CC_\infty$, Angl\`es and Pellarin expand the function $\omega$ about $t = \zeta$ as follows
\begin{equation} \label{omegaaboutzeta}
\omega(t) = \sum_{n \geq 0} \omega^{(n)}(\zeta) \cdot (t - \zeta)^n \in \CC_\infty[[t - \zeta]].
\end{equation}
They prove \cite[Theorem 3.3]{APinv} that for all $n \geq 0$, one has
\[ K(\omega^{(n)}(\zeta)) = K(\zeta, \exp_C(\pitilde/\pfrak^{n+1})) = K_n(\zeta), \]
in particular: the elements $\omega^{(n)}(\zeta)$ are all algebraic over $K$. The $n = 0$ case is especially interesting. Angl\`es and Pellarin prove \cite[Theorem 2.9]{APinv} that $\omega(\zeta)$ is $(\frac{d}{d\theta}\pfrak)(\zeta)^{-1}$ times a basic Gauss-Thakur sum for the $\FF_q$-algebra map on $A$ determined by $\theta \mapsto \zeta$, which is itself a positive characteristic valued Dirichlet character.

Now let $\zeta_1=\zeta,\zeta_2=\zeta^q,\dots,\zeta_d=\zeta^{q^{d-1}}$ be all the distinct roots of $\pfrak$; hence, $\pfrak$ has degree $d$ in $\theta$. The connection made by Angl\`es-Pellarin between $\omega(\zeta_i)$ and the basic Gauss-Thakur sums allows one to conclude that the digit products
\begin{equation} \label{APomegadb} \left\{ \omega(\zeta_1)^{e_1} \omega(\zeta_2)^{e_2} \dots \omega(\zeta_d)^{e_d} : \begin{array}{l}
0 \leq e_1,\dots,e_d \leq q-1, \text{ and } \\ e_j \neq q-1, \text{for some } j \end{array}  \right\}
\end{equation}
form an integral basis for the extension $K(\zeta, \exp_C(\pitilde/\pfrak)) / K(\zeta)$. Further, the monomial $\omega(\zeta_1)^{e_1} \omega(\zeta_2)^{e_2} \dots \omega(\zeta_d)^{e_d}$ lies in the isotypic component for the Dirichlet character on $A$ given by
\[ a \mapsto a(\zeta_1)^{e_1} a(\zeta_2)^{e_2} \dots a(\zeta_d)^{e_d} = a(\zeta)^{\sum e_i q^i}; \]
here we identify the Galois group of $K(\zeta, \exp_C(\pitilde/\pfrak)) / K(\zeta)$ with $(A/\pfrak A)^\times$ as in~\eqref{eq:galois-action}.

\subsection{New results}
The main result of this paper is the generalization of Angl\`es-Pellarin's $\omega$ digit basis \eqref{APomegadb} for $K(\zeta, \exp_C(\pitilde/\pfrak)) / K(\zeta)$ to an $\omega$ digit derivative basis for $K(\zeta, \exp_C(\pitilde/\pfrak^{n+1}))$ over $K(\zeta)$.

\begin{thm*}
Let $\pfrak$ be a monic irreducible of $A$ all of whose distinct roots are $\zeta_1,\dots,\zeta_d$.
The following set provides an integral basis for the extension $K(\zeta, \exp_C(\pitilde/\pfrak^{n+1})) /$ $K(\zeta)$:
\[ \left \{ \prod_{j = 0}^n \prod_{i = 1}^d \omega^{(j)}(\zeta_i)^{e_{j,i}} : \begin{array}{l} 0 \leq e_{j,i} \leq q-1, \text{ for all } 0 \leq j \leq n \text{ and } 1 \leq i \leq d, \text{ and } \\ e_{0,i} \neq q-1, \text{ for some } 0 \leq i \leq d. \end{array}  \right\}. \]
\end{thm*}
A more detailed statement and proof appears in Theorem \ref{thm:integral-basis} below. Of course, digit products come up often in the arithmetic of the Carlitz module, and the results of Conrad \cite{CKdp} were an inspiration for the statement given above.
Our interest in proving the theorem above is the potential utility in further understanding Taelman's class and unit modules, generalizing the work of Angl\`es-Taelman \cite{ATplms} and in this direction we briefly mention in \S \ref{matrixlvalssection} some matrix $L$-values naturally arising from an identity of Pellarin.

Another goal of this short note is to explicate \cite[Section 3.1]{APinv} wherein the Galois action on the elements $\omega^{(n)}(\zeta)$ is briefly discussed.
Again, identifying the Galois group of $K(\zeta, \exp_C(\pitilde/\pfrak^{n+1})) / K(\zeta)$ with $(A/\pfrak^{n+1}A)^\times$ via
\[a \mapsto \bigl(\sigma_a : \exp_C(\pitilde/\pfrak^{n+1}) \mapsto \exp_C(\pitilde a/\pfrak^{n+1})\bigr) \]
we will show that, for each $a \in (A/\pfrak^{n+1}A)^\times$, we have
\[ \sigma_a * \left( \begin{matrix} \omega^{(n)} \\ \omega^{(n-1)} \\ \vdots \\ \omega  \end{matrix} \right)_{t=\zeta} = \left( \begin{matrix} a  & a^{(1)} & \cdots & a^{(n)} \\ 0 & a & \ddots & \vdots \\ \vdots & \ddots & \ddots & a^{(1)} \\ 0 & \cdots & 0 & a \end{matrix} \right)_{\theta = \zeta} \left( \begin{matrix} \omega^{(n)} \\ \omega^{(n-1)} \\ \vdots \\ \omega  \end{matrix} \right)_{t=\zeta}; \]
here $\sigma_a$ acts component-wise on the vector, the matrix subscript denotes evaluation of each entry as specified, and, as before,
\[a = \sum_{j \geq 0} a^{(j)}(\zeta)\cdot(\theta - \zeta)^j \in \FF_q(\zeta)[\theta - \zeta].\]
Thus, as one readily shows, the $\FF_q(\zeta)$-linear subspace of $K_n(\zeta)$ generated by $\omega(\zeta)$, $\omega^{(1)}(\zeta),\dots$, $\omega^{(n)}(\zeta)$ provides a faithful representation of $(A/\pfrak^{n+1}A)^\times$ over $\FF_q(\zeta)$ of minimal dimension.
Further, easy relations coming from the difference equation \eqref{omegadiffeq} satisfied by $\omega$ will allow us to deduce that each basis monomial $ \prod_{j = 0}^n \prod_{i = 1}^d \omega^{(j)}(\zeta_i)^{e_{j,i}}$ --- as in the theorem above --- is killed, for all $a \in (A/\pfrak^{n+1}A)^\times$, by $(\sigma_a - a(\zeta)^k)^m$, for some sufficiently large $m$ and some $k$ depending on the $e_{j,i}$ but not on $a$. In other words, each such monomial lies in a generalized eigenspace for the character $a(\zeta)^l$, again for suitable $l$ easily determined by the $e_{j,i}$ in $ \prod_{j = 0}^n \prod_{i = 1}^d \omega^{(j)}(\zeta_i)^{e_{j,i}}$. Thus the basis of the theorem above, suitably ordered, provides a nice upper-triangular and block diagonal representation of the Galois action that we expect to be useful in the further study of the arithmetic of the extensions $A_n / A$.

\subsubsection*{Acknowledgements} Both authors thank B. Angl\`es for several interesting remarks and questions related to this note.

\section{Generalities}
\subsection{Notations: $\zeta, \pfrak, d, F$}
Throughout we shall fix $\zeta$ in an algebraic closure $\FF_q^{ac} \subset \CC_\infty$ of $\FF_q$, and we let $\pfrak \in A$ be its minimal polynomial, and $d=\deg(\pfrak)$ its degree.
We write $\zeta_1 := \zeta$ and $\zeta_{i+1} = \zeta_i^q$, for $1 \leq i \leq d-1$; of course, then $\zeta_1, \dots, \zeta_d$ are all the roots of $\pfrak$.
At some places we will use the isomorphism $A/\pfrak A \rightarrow \FF_q(\zeta) \subset \CC_\infty$ given by $\theta \mapsto \zeta$.

Finally, we introduce the notation $F$ for the generator of the Galois group
\begin{equation}
\Gal(\FF_q(\zeta) / \FF_q) = \langle F \rangle,
\end{equation}
such that $F(\zeta) = \zeta^q$, and we shall abuse notation by writing $F$ for the generator of any extension $L(\zeta) / L$ isomorphic to $\FF_q(\zeta) / \FF_q$.

We will write $R^{n \times n}$ for the ring of $n\times n$ matrices with coefficients in the ring $R$. If we have an action of a group, etc... on $R$ we shall write $\sigma * T = (\sigma T_{ij})$ for $\sigma$ in the group and $T = (T_{ij}) \in R^{n \times n}$.

\subsection{The Tate algebra} \label{Tatealgstuff}
In this note, we deal exclusively with the {\it Tate algebra} $\TT$ over $\CC_\infty$ in one indeterminate $t$. We recall that
\[ \TT  = \left\{ \sum_{i \geq 0} c_i t^i \in \CC_\infty[[t]] : c_i \rightarrow 0 \text{ as } i \rightarrow \infty \right\} .\]

For elements $\phi = \sum c_j t^j \in \TT$, we define the {\it $i$-th Anderson twist} of $\phi$
\[\phi^{\tau^i} := \sum \tau^i(c_j) t^j= \sum c_j^{q^i} t^j,\]
and extend these $\CC_\infty$-linearly to the (non-commutative) twisted polynomial ring $\CC_\infty\{\tau\}$. We write
\[ \phi^{\ffrak} := \sum_{i \geq 0} \ffrak(c_i) t^i\]
for the action of $\ffrak \in \CC_\infty\{\tau\}$ on $\phi \in \TT$.

We also employ the $\CC_\infty$-algebra map $\mathcal{D} : \TT \rightarrow \TT[[X]]$
\[\phi(t) \mapsto \phi(t+X) := \sum_{n \geq 0} \phi^{(n)}(t) X^n \]
given by replacing the variable $t$ in the power series expansion for $\phi$ by $t+X$, expanding each $(t+X)^n$ using the binomial theorem, and rearranging to obtain a power series in $X$. One readily checks that the coefficients of this power series in $X$ are again in the Tate algebra.
The function $\phi^{(n)}$ is called the {\it $n$-th hyperderivative of $\phi$}. We also note that the maps $\cdot^{(n)} : \TT \rightarrow \TT$, which arise from the algebra map just defined, are $\CC_\infty$-linear and satisfy the Leibniz rule
\[ (\phi\psi)^{(n)} = \sum_{j = 0}^n \phi^{(j)} \psi^{(n-j)},\]
among other familiar (but suitably modified) calculus rules.

We obtain the following family (in $n \geq 1$) of faithful representations of $\CC_\infty$-algebras:
\begin{equation} \label{rhondef} \rho^{[n]} : \TT \rightarrow \TT^{n \times n}, \text{ defined by } \phi \mapsto \left( \begin{matrix} \phi & \phi^{(1)} & \cdots & \phi^{(n-1)} \\ 0 & \phi & \ddots & \vdots \\ \vdots & \ddots & \ddots & \phi^{(1)} \\ 0 & \cdots & 0 & \phi \end{matrix} \right),
\end{equation}
arising from the map $\mathcal{D}$ by evaluation of $X$ at the obvious $n \times n$ nilpotent matrix.

We have the $\FF_q$-algebra embedding
\[ A \xrightarrow{(\cdot)_t} \TT; \ \text{ determined by } \ (\theta)_t = t. \]
When not mentioning the variable $\theta$ we will simply write\footnote{Previous notations for $(\cdot)_t$ have been $\chi_t$ and $a \mapsto a(t)$. } $a_t \in \TT$ for the image of $a \in A$ under $(\cdot)_t$.
For $a \in A$, we may abuse notation by writing $a^{(n)}$ for $a_t^{(n)}$, as in the introduction.

\subsubsection{}

Elements of the Tate algebra may be evaluated at $z \in \CC_\infty$ such that $|z| \leq 1$, and we will write $\rho_z^{[n]} : \TT \rightarrow \CC_\infty^{n \times n}$ for the map
\begin{equation} \label{rhonevdef}
\phi \mapsto \rho^{[n]}(\phi)|_{t = z}. \end{equation}
Clearly, if $\phi(z) \neq 0$, then $\rho^{[n]}_z(\phi)$ is invertible.

Evaluation and twisting interact in the following very nice manner. Let $\pfrak \in A$, $d \geq 1$, and $\zeta_1, \dots, \zeta_d \in \FF_q^{ac}$, be as above.

\begin{lem}[{\bf Evaluation at roots of unity and twisting}] \label{twistevallem} \
\begin{enumerate}
\item For $\phi \in \TT$, there are unique $f_0,\dots, f_{d-1}\in \CC_\infty$ such that
for each $k = 1,\dots,d$, one has
\begin{equation} \phi|_{t = \zeta_k} = \sum_{l = 0}^{d-1} f_l \zeta_k^l.  \label{TTrtof1eval} \end{equation}
\item \label{item:coeffs} Let $\phi=\sum_{i=0}^\infty \phi_it^i\in \TT\subseteq \CC_\infty[[t]]$, and  for each $j = 0,1,\dots, q^d - 2$, write
\[\zeta^j =: \sum_{l = 0}^{d-1} a_{j,l} \zeta^l \in \FF_q[\zeta].\
\footnote{Here we see that the coefficients $a_{j,l}$ are independent of the choice of root $\zeta$ of $\pfrak$ through the action of the Frobenius, which fixes $\FF_q$.} \]
Then the coefficients $f_l$ from equation \eqref{TTrtof1eval} are given as
\[  f_l= \sum_{j = 0}^{q^d - 2}  a_{j,l} \sum_{i \equiv j\! \pmod{q^d - 1}} \phi_i. \]
\item For all $\ffrak \in \CC_\infty\{\tau\}$, with $\phi$ written as in \eqref{TTrtof1eval} above, we have
\[ \phi^{\ffrak}|_{t = \zeta_k} = \sum_{l = 0}^{d-1} \ffrak(f_l) \zeta_k^l. \]
\end{enumerate}
\end{lem}

\begin{proof}
The condition $\phi(\zeta_k)= \phi|_{t = \zeta_k} = \sum_{l = 0}^{d-1} f_l \zeta_k^l$ for all $k = 1,\dots,d$ is equivalent to
\[ \begin{pmatrix}
\phi(\zeta_1) \\ \phi(\zeta_2) \\ \vdots \\ \phi(\zeta_d)
\end{pmatrix} =\begin{pmatrix} 1 & \zeta_1 & \zeta_1^2 & \cdots & \zeta_1^{d - 1} \\ 1 & \zeta_2 & \zeta_2^2 & \cdots & \zeta_2^{d - 1} \\ \vdots \\ 1 & \zeta_d & \zeta_d^2 & \cdots & \zeta_d^{d - 1} \end{pmatrix}  \begin{pmatrix} f_0\\ f_1\\ \vdots \\ f_{d-1} \end{pmatrix}. \]
The square $d \times d$ matrix is of Vandermonde type, hence invertible,
and we obtain unique existence of the $f_k$ by inverting this matrix.

For the second part, one computes for any $k\in \{1,\dots, d\}$ using $\zeta_k^{q^d-1}=1$ :
\[\phi(\zeta_k) =\sum_{i \geq 0} \phi_i \zeta_k^i
= \sum_{j = 0}^{q^d - 2} \zeta_k^j \sum_{i \equiv j \pmod{q^d - 1}} \phi_i = \sum_{l = 0}^{d-1} \zeta_k^l \sum_{j = 0}^{q^d - 2}  a_{j,l} \sum_{i \equiv j \pmod{q^d - 1}} \phi_i. \]
Since, the sums $\sum_{j = 0}^{q^d - 2}  a_{j,l} \sum_{i \equiv j \pmod{q^d - 1}} \phi_i$ are independent of $k$, uniqueness of the expression in \eqref{TTrtof1eval} yields the desired identity.

For proving the third part, let $g_l\in \CC_\infty$ such that
\[ \phi^{\ffrak}|_{t = \zeta_k}=\sum_{l = 0}^{d-1} g_l \zeta_k^l\] for all $k = 1,\dots,d$.
Since, $ \phi^{\ffrak}=\sum_{i \geq 0} \ffrak(\phi_i) t^i$, and
since $\ffrak$ is $\FF_q$-linear and continuous, we obtain from part \ref{item:coeffs}:
\[ g_l= \sum_{j = 0}^{q^d - 2}  a_{j,l} \sum_{i \equiv j \pmod{q^d - 1}} \ffrak(\phi_i)
=\ffrak(\sum_{j = 0}^{q^d - 2}  a_{j,l} \sum_{i \equiv j \pmod{q^d - 1}} \phi_i)=\ffrak(f_l).\]
\end{proof}

\subsubsection{} We introduce the following common multi-index notation for $\bm{e}=(e_1,\dots, e_d)$ and $\bm{f}=(f_1,\dots, f_d)$:
\begin{itemize}
\item $\bm{f}\leq \bm{e}$ means $f_i\leq e_i$ for all $i$, and  $\bm{f}< \bm{e}$ means $\bm{f}\leq \bm{e}$ but not equal.
\item Sums $\bm{f}+ \bm{e}$ and differences  $\bm{f}- \bm{e}$ are componentwise,
\item Binomial coefficients: $\binom{ \bm{e}}{\bm{f}}=\prod_{i=1}^d \binom{e_i}{f_i}$.
\item $\bm{0}=(0,\dots, 0)$ and $\bm{q-1}=(q-1,\dots, q-1)$.
\end{itemize}
Furthermore, with $\zeta_1,\dots,\zeta_d$, as above, for $\phi\in \TT$ and $\bm{e}=(e_1,\dots, e_d)$ we define
\[{\phi}^{\bm{e}}:=\prod_{i=1}^d \phi(\zeta_i)^{e_i}.\]

We notice that for ${\bm 0} \leq {\bm e} \leq \bm{q-1}$ the maps $a\mod \pfrak \mapsto a_t^{\bm e}$ are the monomial functions $A/\pfrak A\cong \FF_q(\zeta)\to \FF_q(\zeta),a(\zeta)\mapsto a(\zeta)^{\sum_{i=0}^{d-1} e_iq^i}$ under the identification of $A/\pfrak A$ with $ \FF_q(\zeta)$ by sending $\theta$ to $\zeta$. Using Lagrange interpolation for maps $\FF_q(\zeta)\to F$ resp.~$\FF_q(\zeta)^\times\to F$ into $\FF_q(\zeta)$-algebras $F$,
 we obtain the following lemma which we also refer to as \textit{Lagrange interpolation}.

\begin{lem}[{\bf Lagrange Interpolation}] \label{lemma:lagrange-interpolation} Let $F$ be an $\FF_q(\zeta)$-algebra.

Any map $f:A/\pfrak A\to F$ can uniquely
be written as
$$f(a)=\sum_{{\bm 0} \leq {\bm e} \leq \bm{q-1}} c_{{\bm e}} a_t^{\bm e}$$
with coefficients $c_{{\bm e}}\in F$.

Similarly, any map $f:(A/\pfrak A)^\times\to F$ can uniquely
be written as
$$f(a)=\sum_{{\bm 0} \leq {\bm e} < \bm{q-1}} c_{{\bm e}} a_t^{\bm e}$$
with coefficients $c_{{\bm e}}\in F$.

In both cases, the coefficients $c_{{\bm e}}\in F$ can be obtained as an $\FF_q(\zeta)$-linear combination of the values of $f$ .
\end{lem}

\section{Results}

\subsection{The Anderson-Thakur function $\omega$: twisting and hyperdifferentiation} \label{omegasection}

We direct the reader back to the introduction where the Anderson-Thakur function $\omega$ was defined and discussed. From \eqref{omegadiffeq}, it follows that
\begin{equation} \label{omegacarlitztwisteq} \omega^{\cfrak_a} = a_t\omega,  \ \forall a \in A.
\end{equation}
Now, hyperdifferentiation and twisting commute. Thus, hyperdifferentiating and using the Leibniz rule, we obtain
\begin{align} \label{omeganfnleqn}
(\omega^{(n)})^{\cfrak_a} = (\omega^{\cfrak_a})^{(n)} = (a_t, a_t^{(1)}, \dots, a_t^{(n)} ) \cdot \left( \begin{matrix} \omega^{(n)} \\ \omega^{(n-1)} \\ \vdots \\ \omega  \end{matrix} \right ).  \end{align}
In fact, one may express \eqref{omeganfnleqn} more compactly with the representation $\rho^{[n+1]}$ of $\TT$ defined in \eqref{rhondef}.
From \eqref{omegacarlitztwisteq}, we obtain
\[ \cfrak_a * \rho^{[n+1]}(\omega) =\rho^{[n+1]}(\omega^{\cfrak_a}) = \rho^{[n+1]}(a_t\omega) = \rho^{[n+1]}(a_t)\rho^{[n+1]}(\omega).\]

Similarly, we obtain the following recursion for $n \geq 1$:
\begin{equation} \label{omeganrecurn}
\omega^{(n)}(t)^q = (t^q - \theta)\omega^{(n)}(t^q) + \omega^{(n-1)}(t^q).
\end{equation}

\subsection{Integrality of \boldmath$\omega^{(n)}(\zeta)$ and Galois action}

In the first part of the proof of \cite[Thm.~3.3]{APinv}, it is shown that $\omega^{(n)}(\zeta)$ is an element of $K_n(\zeta)$ and that it is integral over $A[\zeta]$. In the next result, we give an independent proof of this fact, and demonstrate one way these elements are akin to Gauss sums: they are $\FF_q(\zeta)$-linear combinations of Carlitz torsion.

\begin{prop}[{\bf Integrality of \boldmath$\omega^{(n)}(\zeta)$}] \label{prop:omega-n-integral}
As above, let $\zeta=\zeta_1,\zeta_2,\ldots, \zeta_d \in \FF_q^{ac} \subset \CC_\infty$ be the roots of the monic irreducible polynomial $\pfrak \in A$ of degree $d$.
Write
\begin{equation} \label{torsioncoeffseq}
\omega^{(n)}(\zeta) := \sum_{i = 0}^{d - 1} c_{(n),i} \zeta^{i}.
\end{equation}
as in \eqref{TTrtof1eval}.
We have $c_{(n),i} \in C[\pfrak^{n+1}]$, for all $i = 0,1,\dots, d - 1$. Hence,
\[\omega^{(n)}(\zeta) \in A_n[\zeta].\]
\end{prop}

\begin{proof}
From Lemma \ref{twistevallem} and \eqref{omeganfnleqn} above, we see that
\[\sum_{i = 0}^{d - 1} \cfrak_{\pfrak^{n+1}}(c_{(n),i}) {\zeta_k}^{i} = (\omega^{(n)})^{\cfrak_{\pfrak^{n+1}}}|_{t = \zeta_k} = 0,\]
for each $k = 1, \dots, d$.
By uniqueness of the coefficients, we conclude $\cfrak_{\pfrak^{n+1}}(c_{(n),i}) = 0$, for all $i = 0,1,\dots, d - 1$.
\end{proof}

For the next result, see also \cite[Proposition 3.6]{APinv}.
\begin{cor}[{\bf Galois action on \boldmath$\omega^{(n)}(\zeta)$}] \label{galoisactioncor}
For each $\sigma_a \in \Gal(K_n(\zeta) / K(\zeta))$, we have
\[ \sigma_a * \left( \begin{matrix} \omega^{(n)} \\ \omega^{(n-1)} \\ \vdots \\ \omega  \end{matrix} \right)_{t=\zeta} = \left( \begin{matrix} a_t  & a_t^{(1)} & \cdots & a_t^{(n)} \\ 0 & a_t & \ddots & \vdots \\ \vdots & \ddots & \ddots & a_t^{(1)} \\ 0 & \cdots & 0 & a_t \end{matrix} \right)_{t = \zeta} \left( \begin{matrix} \omega^{(n)} \\ \omega^{(n-1)} \\ \vdots \\ \omega  \end{matrix} \right)_{t=\zeta}. \]
This can also be expressed more succinctly in matrix form using $\rho_\zeta^{[n+1]}$ defined in~\eqref{rhonevdef},
\[\sigma_a * \rho_\zeta^{[n+1]}(\omega) = \rho_\zeta^{[n+1]}(a_t) \rho_\zeta^{[n+1]}(\omega), \quad \forall a \in (A/\pfrak^{n+1}A)^\times. \]
Further,
\[ F(\omega^{(n)}(\zeta_i)) = \omega^{(n)}(\zeta_{i+1}), \quad \forall i = 1,\dots,d-1.\]
\end{cor}

\subsubsection{Remark}
Beware that $F(\omega^{(n)}(\zeta)) \neq \omega^{(n)}(\zeta)^q$. Indeed, by \eqref{omeganrecurn}, for the latter we obtain
\[ \omega^{(n)}(\zeta)^q = (\zeta^q - \theta)\omega^{(n)}(\zeta^q) + \omega^{(n-1)}(\zeta^q). \]
This relation can be seen as motivation for the need to use only exponents up to $q-1$ of $\omega^{(n)}(\zeta)$ in the digit derivative basis for $K_n(\zeta) / K(\zeta)$.

\begin{proof}[Proof of Corollary \ref{galoisactioncor}]
We have
\begin{equation} \label{omegagaloiscarlitzeq}
\sigma_a(\omega^{(n)}(\zeta)) = \sum_{i = 0}^{\deg (\pfrak) - 1} \sigma_a(c_{(n),i}) \zeta^{i} = \sum_{i = 0}^{\deg (\pfrak) - 1} \cfrak_a(c_{(n),i}) \zeta^{i}, \end{equation}
where the last equality follows since $c_{(n),i} \in C[\pfrak^{n+1}]$, for all $i = 0,1,\dots, \deg (\pfrak) - 1$. The right side of the previous displayed equation is nothing other than $(\omega^{(n)})^{\cfrak_a}|_{t = \zeta}$, and the result follows by applying \eqref{omeganfnleqn}.

The last claim $F(\omega^{(n)}(\zeta_i)) = \omega^{(n)}(\zeta_{i+1})$ follows since $F$ fixes $C[\pfrak^{n+1}]$, and hence all of the $c_{(n),i}$.
\end{proof}

\subsection{Explicit formulas for the coefficients $c_{(n),i}$ of $\omega^{(n)}(\zeta)$}

The question of using Carlitz $\pfrak^{n+1}$-torsion and roots of unity to explicitly write down non-zero elements of $K_n(\zeta)$ generating a subspace such that the Galois group of $K_n(\zeta) / K(\zeta)$ acts by some power of $\rho^{[n+1]}_\zeta$ leads us to the question to determine the coefficients $c_{(n),i}$ of $\omega^{(n)}(\zeta)$ in \eqref{torsioncoeffseq} explicitly.

The key ingredient for their computation is the identity
\begin{equation} \label{eq:omega-as-exp}
 \omega(t)= \sum_{m\geq 0} \frac{\pitilde^{q^m}}{D_m (\theta^{q^m}-t)},
\end{equation}
given in  \cite[Sect.~4]{FPannals} and relating the function $\omega$ to the Carlitz exponential, and the identities for $\omega^{(n)}(t)$ obtained by hyperdifferentiating this identity.

We start with the explicit description of the coefficients of $\omega(\zeta)$.

\begin{prop}
Let  $\pfrak=\sum_{j=0}^d a_j\theta^j\in A$.

The coefficients $c_{(0),i}$ from \eqref{torsioncoeffseq} are given by
\[ c_{(0),i}= \cfrak_{\qfrak_{(0),i}}\left( \exp_C(\pitilde/\pfrak)\right)= \exp_C\left(\pitilde \qfrak_{(0),i}/\pfrak\right), \]
where $\qfrak_{(0),i} =\sum_{j=i+1}^d a_j\theta^{j-i-1}\in A$.

In particular, $\qfrak_{(0),d-1}=a_d=1$, and $ c_{(0),d-1}= \exp_C(\frac{\pitilde}{\pfrak})$.
\end{prop}

\begin{proof}
By construction, the $c_{(0),i}$ are the coefficients of the polynomial $c_{(0)}(T)\in \CC_\infty[T]$ of degree $<d$ satisfying
$c_{(0)}(\zeta_j)=\omega(\zeta_j)$ for all $1\leq j\leq d$. By classical Lagrange interpolation,

\[ c_{(0)}(T)=\sum_{j=1}^d \omega(\zeta_j) \prod_{k\ne j}\frac{T-\zeta_k}{\zeta_j-\zeta_k}
= \sum_{j=1}^d \frac{\omega(\zeta_j)}{\pfrak'(\zeta_j)}\cdot \frac{\pfrak(T)}{(T-\zeta_j)} \]

Hence by \eqref{eq:omega-as-exp},
\begin{eqnarray*}
c_{(0)}(T)
&=&  \sum_{m\geq 0} \frac{\pitilde^{q^m}}{D_m} \sum_{j=1}^d \frac{1}{\pfrak'(\zeta_j)(\theta^{q^m}-\zeta_j)} \frac{\pfrak(T)}{(T-\zeta_j)}\\
&=&  \sum_{m\geq 0} \frac{\pitilde^{q^m} \qfrak_{(0)}(\theta^{q^m},T)}{D_m \pfrak(\theta^{q^m})},
\end{eqnarray*}
where
\[ \qfrak_{(0)}(\theta,T):= \sum_{j=1}^d \frac{1}{\pfrak'(\zeta_j)}\frac{\pfrak(\theta)}{(\theta-\zeta_j)}
\frac{\pfrak(T)}{(T-\zeta_j)}\in A[T]. \]
It therefore remains to show that the coefficient $\qfrak_{(0),i}$ of $T^i$ in $\qfrak_{(0)}(\theta,T)$ is of the given form.

As for all $k=1,\ldots d$:
\[ \qfrak_{(0)}(\theta,\zeta_k)=\frac{\pfrak(\theta)}{\theta-\zeta_k}=\frac{\pfrak(\theta)-\pfrak(\zeta_k)}{\theta-\zeta_k}
=\frac{\pfrak(\theta)-\pfrak(T)}{\theta-T}|_{T=\zeta_k}, \]
and $\qfrak_{(0)}(\theta,T)$ as well as $\frac{\pfrak(\theta)-\pfrak(T)}{\theta-T}$ are polynomials in $T$ of degree $<d$, we have
\[ \qfrak_{(0)}(\theta,T)=\frac{\pfrak(\theta)-\pfrak(T)}{\theta-T}. \]
Writing $\pfrak=\sum_{j=0}^d a_j\theta^j$, we finally get
\begin{eqnarray*}
 \qfrak_{(0)}(\theta,T) &=& \frac{\sum_{j=0}^d a_j (\theta^j-T^j)}{\theta-T}
 = \sum_{j=0}^d a_j \sum_{i=0}^{j-1} \theta^{j-1-i}T^i \\
 &=& \sum_{i=0}^{d-1} \left( \sum_{j=i+1}^{d} a_j\theta^{j-1-i} \right) T^i.
\end{eqnarray*}
\end{proof}

\begin{prop}
For each $n \geq 1$, let
\[\qfrak_{(n)}(\theta,T):= \sum_{j=1}^d \frac{1}{\pfrak'(\zeta_j)}\frac{\pfrak(\theta)^{n+1}}{(\theta-\zeta_j)^{n+1}} \frac{\pfrak(T)}{(T-\zeta_j)} = \sum_{j \geq 0} \qfrak_{(n),j}T^j \in A[T] \]
be the unique polynomial of degree $<d$ in $T$ interpolating the map $\zeta_j \mapsto \frac{\pfrak(\theta)^{n+1}}{(\theta-\zeta_j)^{n+1}}$.

We have
 \[ c_{(n),i}=\cfrak_{\qfrak_{(n),i}}\left( \exp_C(\frac{\pitilde}{\pfrak^{n+1}})\right)= \exp_C\left(\frac{\pitilde \qfrak_{(n),i}}{\pfrak^{n+1}}\right). \]
\end{prop}

\begin{proof}
The proof is similar to the case $n=0$.

The $c_{(n),i}$ are the coefficients of the unique polynomial $c_{(n)}(T)\in \CC_\infty[T]$ satisfying
$\deg_T(c_{(n)}(T)) < d$ and $c_{(n)}(\zeta_j)=\omega^{(n)}(\zeta_j)$, for all $1\leq j\leq d$. This gives
\[ c_{(n)}(T)=\sum_{j=1}^d \omega^{(n)}(\zeta_j) \prod_{k\ne j}\frac{T-\zeta_k}{\zeta_j-\zeta_k}
= \sum_{j=1}^d \frac{\omega^{(n)}(\zeta_j)}{\pfrak'(\zeta_j)}\frac{\pfrak(T)}{(T-\zeta_j)}. \]
By hyperdifferentiating \eqref{eq:omega-as-exp}, one has
\[  \omega^{(n)}(t)= \sum_{m\geq 0} \frac{\pitilde^{q^m}}{D_m (\theta^{q^m}-t)^{n+1}}, \]
and hence, by an analogous computation to the case $n=0$ above,
\[  c_{(n)}(T) = \sum_{m\geq 0} \frac{\pitilde^{q^m}}{D_m \pfrak(\theta^{q^m})^{n+1}} \qfrak_{(n)}(\theta^{q^m},T). \]

\end{proof}

\begin{rem}
Notice that for all $k=1,\ldots d$:
\[ \qfrak_{(n)}(\theta,\zeta_k)= \frac{\pfrak(\theta)^{n+1}}{(\theta-\zeta_k)^{n+1}}=\frac{(\pfrak(\theta)-\pfrak(\zeta_k))^{n+1}}{(\theta-\zeta_k)^{n+1}}
=\left( \frac{\pfrak(\theta)-\pfrak(T)}{\theta-T}\right)^{n+1}|_{T=\zeta_k}. \]
Hence, $ \qfrak_{(n)}(\theta,T)$ and $\left( \frac{\pfrak(\theta)-\pfrak(T)}{\theta-T}\right)^{n+1}$ have to be congruent modulo $\pfrak(T)$, giving a connection to the $n = 0$ case.
\end{rem}

\begin{rem}
It appears difficult to obtain the Galois action on $\omega^{(n)}(\zeta) = \sum_{j = 0}^{d-1} c_{(n),j} \zeta^j$ directly from the explicit description of the $c_{(n),j}$ given above. Nevertheless, such an explicit description of these coefficients can be useful for implementing these objects on the computer.
\end{rem}

\subsection{Matrix $L$-values} \label{matrixlvalssection}

Another way to obtain explicit elements using Carlitz $\pfrak^{n+1}$-torsion and roots of unity so that the Galois action on them is given by some power of $\rho^{[n+1]}_\zeta$ can be made through the connection between $\omega$ and a special $L$-value made by Pellarin \cite[Theorem 1]{FPannals}, namely the identity
\begin{equation} \label{pellarinLid} \frac{-1}{\pitilde} (t - \theta) L(\chi_t,1) \omega(t)  = 1. \end{equation}

Let
\[ L(\chi_t,1) := \sum_{\substack{a \in A \\ a \text{ monic}}} \frac{a_t}{a} = \prod_{\substack{\qfrak \in A \\ \qfrak \text{ monic,irred}}} (1 - \frac{\qfrak_t}{\qfrak})^{-1} \in \TT \]
be the special $L$-value of Pellarin.
Applying $\rho^{[n+1]}$ to $L(\chi_t,1)$, we observe that it respects both the $A$-harmonic sum and Euler product. So $\rho^{[n+1]}(L(\chi_t,1))$ is a kind-of matrix $L$-value; see \cite[\S 2.2.2]{FPreps} where more general such $L$-values have already been observed.

Fixing $n$ and $\zeta$, we let
\[ {\bm L}^{[n+1]}_\zeta := \frac{\pfrak^{n+1}}{\pitilde} \rho_\zeta^{[n+1]}(L(\chi_t,1)) \in \GL_{n+1}(K_n(\zeta)).\]
The explicit entries of $\frac{\pitilde}{\pfrak^{n+1}} {\bm L}^{[n+1]}_\zeta$ are the elements
\[L^{(j)}_\zeta  := \sum_{a \in A_+} \frac{a_t^{(j)}(\zeta)}{a},\]
which in the phraseology of Anderson are a kind-of ``twisted $A$-harmonic series.''

\begin{prop}[{\bf {L}-matrix}]
All entries of ${\bm L}^{[n+1]}_\zeta$ are integral, except possibly $\frac{\pfrak^{n+1}}{\pitilde} L^{(n)}_\zeta$, for which one has $\frac{\pfrak^{n+2}}{\pitilde} L^{(n)}_\zeta \in A_n[\zeta]$.

Further, we have
\[ {\bm L}^{[n+1]}_\zeta = -\sideset{}{'}\sum_{ a \in A/\pfrak^{n+1}A } \rho^{[n+1]}_\zeta(a_t)\frac{1}{\exp_C(\frac{\pitilde a}{\pfrak^{n+1}})}. \]
Hence,
\[\sigma_a * {\bm L}^{[n+1]}_\zeta = \rho_\zeta^{[n+1]}(a_t)^{-1} {\bm L}^{[n+1]}_\zeta, \quad \forall a \in (A/\pfrak^{n+1} A)^\times. \]
\end{prop}

\begin{proof}
Since the hyperderivatives $\cdot^{(j)}$ are $\FF_q$-linear on $A$, the sum defining $L^{(j)}_\zeta$ can be extended to one over all non-zero elements of $A$ up to a minus sign:
\[ L^{(j)}_\zeta = -\sideset{}{'}\sum_{ a \in A} \frac{a_t^{(j)}(\zeta)}{a}.\]
Now, for each $j = 0,1,\dots,n$ and $a \in \pfrak^{n+1}A$, we have
\[ a_t^{(j)}(\zeta) = 0. \]
Thus, ``collapsing the sum'' we obtain
\begin{align*} L^{(j)}_\zeta &= -\sum_{0 \neq a \in A/\pfrak^{n+1}A} a_t^{(j)}(\zeta) \sum_{b \in A} \frac{1}{a+b\pfrak^{n+1}} \\
&= -\frac{\pitilde}{\pfrak^{n+1}} \sum_{0 \neq a \in A/\pfrak^{n+1}A} \frac{a_t^{(j)}(\zeta)}{\exp_C(\frac{\pitilde a}{\pfrak^{n+1}})}. \end{align*}
By these calculations, we deduce the desired matrix identity
\[ {\bm L}^{[n+1]}_\zeta = -\sideset{}{'}\sum_{ a \in A/\pfrak^{n+1}A } \rho^{[n+1]}_\zeta(a_t)\frac{1}{\exp_C(\frac{\pitilde a}{\pfrak^{n+1}})}. \]

From the case $j = n$, we obtain
\begin{equation} \label{Ljfinitesumeq}
-\frac{\pfrak^{j+1}}{\pitilde} L^{(j)}_\zeta = \sum_{0 \neq a \in A/\pfrak^{j+1}A} \frac{a^{(j)}(\zeta)}{\exp_C(\frac{\pitilde a}{\pfrak^{j+1}})} \in K_j(\zeta), \end{equation}
demonstrating that these elements lie in successively smaller extensions.
Indeed, from \eqref{Ljfinitesumeq}, we observe that $\frac{\pfrak^{j+2}}{\pitilde} L^{(j)}_\zeta \in A_j[\zeta]$, for all $j$ (since $\frac{\pfrak}{\exp_C(\frac{\pitilde a}{\pfrak^{n+1}})}$ is always integral), and hence all entries of ${\bm L}^{[n+1]}_\zeta$ are integral, except possibly the top entry $\frac{\pfrak^{n+1}}{\pitilde} L^{(n)}_\zeta$.
\end{proof}

\subsubsection{Remark}
As B. Angl\`es points out, for applications to understanding Taelman's class and unit modules for $A_n$, one should consider the family of special values given by
\[ \sum_{a \in A_+} \frac{\prod_{j = 0}^n a_t^{(j)}(\zeta)^{e_j}}{a}, \quad 0 \leq e_j \leq |\pfrak| - 2. \]
Perhaps these may be easily related to Taelman units for $A_n$ via the observations of \cite[\S 2.2.3]{FPreps} and the formalism of evaluation at characters of \cite[\S 9]{APTR}.
We hope to return to these investigations in a future work.

\subsection{Valuations of \boldmath$\omega^{(n)}(\zeta)$ at primes above \boldmath$\pfrak$}

As $\zeta$ is a root of $\pfrak$, the prime $(\pfrak)$ splits completely in $K(\zeta)$ into the primes
$(\theta-\zeta_i)$ for $i=1,\dots, d$, where as before $\zeta_i=\zeta^{q^{i-1}}$.
The valuations corresponding to $(\theta-\zeta_i)$ will be denoted by $v_i$.
As these primes ramify totally in the extension $K_n(\zeta) / K(\zeta)$, we will also denote by $v_i$, the $\QQ$-valued extension of these valuations to $K_n(\zeta)$.

The following result provides the valuations of the $\omega^{(n)}(\zeta)$ at these primes.
Another demonstration of this result can be found in the proof of \cite[Thm.~3.3]{APinv} where this fact is also stated (cf.~equation (23) ibid.).

\begin{thm}\label{thm:valuations-of-omega-n}
 For all $n\geq 0$, $1\leq i\leq d$ and $0\leq j\leq d-1$ we have
\[ v_i\left( \omega^{(n)}(\zeta_i^{q^j})\right) =\frac{q^j}{|\pfrak|^n(|\pfrak|-1)}.
\]
\end{thm}

\begin{proof}
We do this by induction on $n$, following the lines of \cite[Proposition 2.1]{APinv}.
Using the equation $\omega^\tau=(t-\theta)\omega$, one easily obtains that
$\omega(\zeta_i^{q^j})$ is a root of the polynomial equation
\[ X^{|\pfrak|-1}-\beta(\zeta_i^{q^j}), \]
where $\beta=\prod_{h=0}^{d-1} (t-\theta^{q^h})\in \TT$.
Since
\[ v_i\left(\beta(\zeta_i^{q^j})\right)=\sum_{h=0}^{d-1} v_i\left(\zeta_i^{q^j}-\theta^{q^h}\right)= q^j,\] the Newton polygon of the above equation has exactly on slope,
namely $\frac{q^j}{|\pfrak|-1}$. Hence, the valuation of $\omega(\zeta_i^{q^j})$ is the desired one.

Similarly for $n>0$, one obtains that $\omega^{(n)}(\zeta_i^{q^j})$ is a root of the polynomial
\[ X^{|\pfrak|}-\beta(\zeta_i^{q^j})X-\xi_n(\zeta_i^{q^j})\in K_{n-1}(\zeta)[X], \]
 where
\[ \xi_n(t)= \sum_{l=1}^n \beta^{(l)}\omega^{(n-l)}. \]
Now, $\beta^{(1)}(\zeta_i^{q^j})=\sum_{h=0}^{d-1} \prod_{h'\ne h} \left(\zeta_i^{q^j}-\theta^{q^{h'}}\right)$ has valuation $0$  and
\[ v_i\left( \omega^{(n-1)}(\zeta_i^{q^j})\right)=\frac{q^j}{|\pfrak|^{n-1}(|\pfrak|-1)}< v_i\left( \omega^{(n-l)}(\zeta_i^{q^j})\right) \]
for all $l\geq 2$ by induction hypothesis. Hence,
\[ v_i\left(\xi_n(\zeta_i^{q^j})  \right)=v_i\left( \beta^{(1)}(\zeta_i^{q^j})\omega^{(n-1)}(\zeta_i^{q^j})\right)
= \frac{q^j}{|\pfrak|^{n-1}(|\pfrak|-1)}. \]
Therefore, using again $v_i\left(\beta(\zeta_i^{q^j})\right)=q^j$, the Newton polygon has exactly one slope, and this is
$\frac{1}{|\pfrak|}\cdot  \frac{q^j}{|\pfrak|^{n-1}(|\pfrak|-1)}= \frac{q^j}{|\pfrak|^{n}(|\pfrak|-1)},$
 giving the desired valuation for $\omega^{(n)}(\zeta_i^{q^j})$.
 \end{proof}

\begin{rem}
The proof of the previous theorem even shows more than the stated fact. Indeed,  the occurring polynomials for $\omega(\zeta_i^{q^j})$ over $K(\zeta)$ and for $\omega^{(n)}(\zeta_i^{q^j})$ over $K_{n-1}(\zeta)$ ($n\geq 1$) are irreducible, since their Newton polygons have only one slope. Hence they are the minimal polynomials of these elements.

In particular, one deduces that the monomials $\omega^{(n)}(\zeta)^j$, with $0 \leq j \leq |\pfrak| - 2$ (respectively, with $0 \leq j \leq |\pfrak| - 1$), give a field basis for $K_n(\zeta) / K_{n-1}(\zeta)$ when $n = 0$ (resp. when $n \geq 1$); here $K_{-1}(\zeta) := K(\zeta)$. We will refine this observation in the next section.
\end{rem}

\begin{rem}
We expect, for $\deg(\pfrak) \geq 2$ and $n\geq 1$, that there are finite primes other than those above $(\theta - \zeta_i)A[\zeta]$ which divide $\omega^{(n)}(\zeta_i)$. Consideration of the sum of the readily computable $\infty$-adic valuations of $\omega^{(n)}(\zeta)$ and the $(\theta-\zeta_i)$-adic valuations computed above should lead to a verification of this claim.
\end{rem}

\subsection{An integral basis for the prime power extensions} \label{sec:integral-basis}

We remind the reader of the multi-index notation introduced in \S \ref{Tatealgstuff}.

\begin{prop} \label{prop:basis-for-extension}
\begin{enumerate}
\item
The elements ${\omega}^{\bm{e}}$ with $\bm{0}\leq \bm{e}< \bm{q-1}$ form a basis of $K_0(\zeta)$ over $K(\zeta)$.
Further for all $n>0$, the elements $({\omega^{(n)}})^{\bm{e}}$ with $\bm{0}\leq \bm{e}\leq \bm{q-1}$ form a basis of $K_n(\zeta)$ over $K_{n-1}(\zeta)$.
\item If any linear combination $x:=\sum_{{\bm e}} c_{{\bm e}} \cdot ({\omega^{(k)}})^{\bm{e}}$ with
$c_{{\bm e}}\in K_{n-1}(\zeta)$ satisfies $v_i(x)\geq 0$, then
$v_i(c_{{\bm e}})\geq 0$ for all $\bm{e}$.
\end{enumerate}
\end{prop}

\begin{proof}
By Theorem \ref{thm:valuations-of-omega-n}, for all $n\geq 0$ one has
\[ v_1\left( ({\omega^{(n)}})^{\bm{e}}\right)=\sum_{i=1}^d e_i\cdot v_1(\omega^{(n)}(\zeta_i))
=\frac{\sum_{i=1}^d e_iq^{i-1}}{|\pfrak|^n(|\pfrak|-1)}, \]
as $\zeta_i=\zeta_1^{q^{i-1}}$.
For $n=0$, we have $\bm{0}\leq \bm{e}< \bm{q-1}$, and so these values are all different and lie between $0$ (included) and $1$ (excluded). Since the valuation $v_1$ is $\ZZ$-valued on $K(\zeta)$, for every non-trivial linear combination $x:=\sum_{{\bm e}} c_{{\bm e}} {\omega}^{\bm{e}}$, we have
\[ v_1(x)=\min_{\bm e} \{ v_1(c_{{\bm e}} {\omega}^{\bm{e}}) \}. \]
In particular, $x\ne 0$. Furthermore, if $v_1(x)\geq 0$, this implies $v_1(c_{{\bm e}})\geq 0$ for all ${\bm e}$.

For fixed $n\geq 1$, we have $\bm{0}\leq \bm{e}\leq  \bm{q-1}$, and so the values $\frac{\sum_{i=1}^d e_iq^{i-1}}{|\pfrak|^n(|\pfrak|-1)}$ are again all different and lie between $0$ (included) and $\frac{1}{|\pfrak|^{n-1}(|\pfrak|-1)}$ (excluded).
Since the valuation $v_1$ is $\frac{1}{|\pfrak|^{n-1}(|\pfrak|-1)}\ZZ$-valued on $K_{n-1}(\zeta)$, for every non-trivial linear combination $x:=\sum_{{\bm e}} c_{{\bm e}} \cdot ({\omega^{(n)}})^{\bm{e}}$, we have
\[ v_1(x)=\min_{\bm e} \{ v_1(c_{{\bm e}} \cdot ({\omega^{(n)}})^{\bm{e}}) \}. \]
In particular, $x\ne 0$.
Furthermore, if $v_1(x)\geq 0$, this implies $v_1(c_{{\bm e}})\geq 0$ for all ${\bm e}$.

By the same argument, we get the second claim also for the other valuations~$v_i$.
\end{proof}

\begin{cor}[{\bf Non-vanishing of torsion coefficients of \boldmath$\omega^{(n)}(\zeta)$}] \label{nonvantorsioncoeffcor}
Let $c_{(n),i}$ be as defined in \eqref{torsioncoeffseq} for $\omega^{(n)}(\zeta)$, for $i = 0, 1,\dots, d-1$. We have
\[ c_{(n),i} \in C[\pfrak^{n+1}] \setminus C[\pfrak^{n}], \quad \forall i = 0,1,\dots,d-1. \]
In particular, these coefficients are non-zero for all $n \geq 0$ and all $i = 0,1,\dots,d-1$.
\end{cor}

\begin{proof}
From \eqref{omeganfnleqn} and Lemma \ref{twistevallem}, with $a = \pfrak^n$, we deduce that
\[ \left( \begin{matrix} \cfrak_{\pfrak^{n}}(c_{(n),0}) \\ \cfrak_{\pfrak^{n}}(c_{(n),1}) \\ \vdots \\ \cfrak_{\pfrak^{n}}(c_{(n),d-1}) \end{matrix}  \right) = \left( \begin{matrix} 1 & \zeta_1 & \zeta_1^2 & \cdots & \zeta_1^{d - 1} \\ 1 & \zeta_2 & \zeta_2^2 & \cdots & \zeta_2^{d - 1} \\ \vdots \\ 1 & \zeta_d & \zeta_d^2 & \cdots & \zeta_d^{d - 1} \end{matrix}  \right)^{-1}\left( \begin{matrix} \pfrak_t^{(1)}(\zeta_1)^n \omega(\zeta_1) \\ \pfrak_t^{(1)}(\zeta_2)^n \omega(\zeta_2) \\ \vdots \\ \pfrak_t^{(1)}(\zeta_d)^n \omega(\zeta_d) \end{matrix}  \right).\]
Thus, if $\cfrak_{\pfrak^{n}}(c_{(n),i}) = 0$, for some $i = 0,1,\dots,d-1$, we obtain a non-trivial linear $A[\zeta]$-dependence relation on the elements $\omega(\zeta_1),\dots,\omega(\zeta_d)$, but by Proposition \ref{prop:basis-for-extension}, these latter elements are linearly independent over $K(\zeta)$. This contradiction establishes the result.
\end{proof}

\begin{thm}[{\bf Integral basis for \boldmath$A_n[\zeta]$ over \boldmath$A[\zeta]$}] \label{thm:integral-basis} \
\begin{enumerate}
\item\label{item:n=0} The elements ${\omega}^{\bm{e}}$ with $\bm{0}\leq \bm{e}< \bm{q-1}$ form an integral basis of $A_0[\zeta]$ over $A[\zeta]$.
\item\label{item:n>0} Further, for all $n>0$, the elements $({\omega^{(n)}})^{\bm{e}}$ with $\bm{0}\leq \bm{e}\leq  \bm{q-1}$ form an integral basis of $A_n[\zeta]$ over $A_{n-1}[\zeta]$.
\item\label{item:integral-basis} Hence, the products $\prod_{j = 0}^n (\omega^{(j)})^{{\bm e}_j}$ with $\bm{0}\leq \bm{e}_0< \bm{q-1}$ and $\bm{0}\leq \bm{e}_j\leq  \bm{q-1}$, for $j = 1,\dots,n$, form an integral basis for $A_{n}[\zeta]$ over $A[\zeta]$.
\end{enumerate}
\end{thm}

\begin{proof}
By Proposition \ref{prop:omega-n-integral}, we already know that the elements $({\omega^{(n)}})^{\bm{e}}$ are integral.
For $n=0$, it is well known that ${\omega}^{\bm{e}}$ form an integral basis as they are $\FF_q(\zeta)^\times$-multiples of Gauss sums; see e.g. \cite[Th\'eor\`eme 2.5]{BAjtnb}.
However, we will give another proof here, since parts of the proof for the elements $({\omega^{(n)}})^{\bm{e}}$ will be in the same flavor.

By Corollary \ref{galoisactioncor}, for all $a\in (A/\pfrak A)^\times$ we have
$ \sigma_a(\omega(\zeta_i))=a_t(\zeta_i)\cdot \omega(\zeta_i)$, and hence
\[ \sigma_a\left( {\omega}^{\bm{e}}\right) = {a_t}^{\bm{e}}  {\omega}^{\bm{e}}. \]
Given $c_{\bm e}\in K(\zeta)$ such that $x:=\sum_{{\bm e}} c_{{\bm e}} {\omega}^{\bm{e}}\in A_0[\zeta]$, we therefore have
\[ \sigma_a( x)= \sum_{{\bm e}} c_{{\bm e}}  {\omega}^{\bm{e}}{a_t}^{\bm{e}}\in A_0[\zeta] \]
for all $a\in (A/\pfrak A)^\times$.
By Lagrange interpolation \ref{lemma:lagrange-interpolation}, we therefore get $c_{{\bm e}}  {\omega}^{\bm{e}}\in  A_0[\zeta]$ for all
${\bm e}$. As ${\omega}^{\bm{e}}$ only has zeros at the primes above $\pfrak$, this implies that the $c_{{\bm e}}$ can only have poles at these primes. In Proposition \ref{prop:basis-for-extension}, however, we already showed that the $c_{{\bm e}}$ do not have poles there. Hence, $c_{{\bm e}}\in A[\zeta]$.

Now let $n\geq 1$. Any $b\in A/\pfrak A$ gives rise to a well-defined $1+\pfrak^{n}b\in (A/\pfrak^{n+1}A)^\times$.
Using Corollary \ref{galoisactioncor} and $\pfrak_t(\zeta_i)=0$, we obtain
\[ \sigma_{1+\pfrak^{n}b}(\omega^{(n)}(\zeta_i))=\omega^{(n)}(\zeta_i)+\left(\pfrak_t^{(1)}(\zeta_i)\right)^nb_t(\zeta_i)\omega(\zeta_i), \]
and hence,
\begin{eqnarray}  \sigma_{1+\pfrak^{n}b}(({\omega^{(n)}})^{\bm{e}})
&=& \left( { \omega^{(n)}}+\bigl({\pfrak_t^{(1)}}\bigr)^n {b_t  \omega } \right)^{\bm{e}} \\
&=& \sum_{\bm{0}\leq\bm{f}\leq \bm{e}} \binom{\bm{e}}{\bm{f}} \left( \bigl({\pfrak_t^{(1)}}\bigr)^n {b_t  \omega } \right)^{\bm{e}-\bm{f}} ({\omega^{(n)}})^{\bm{f}}
\end{eqnarray}
for all $\bm{0}\leq \bm{e}\leq \bm{q-1}$.

Now let $c_{{\bm e}}\in K_{n-1}(\zeta)$, such that
\[x:=\sum_{{\bm e}} c_{{\bm e}} \cdot ({\omega^{(k)}})^{\bm{e}}\in A_{n}[\zeta].\]

For all $b\in A/\pfrak A$, we obtain
\begin{eqnarray}
 \sigma_{1+\pfrak^{n}b}(x) &=& \sum_{\bm{e}}  \sum_{\bm{0}\leq \bm{f}\leq \bm{e}} c_{\bm{e}} \binom{\bm{e}}{\bm{f}} \left( \bigl({\pfrak_t^{(1)}}\bigr)^n {b_t  \omega } \right)^{\bm{e}-\bm{f}} ({\omega^{(n)}})^{\bm{f}}\\
 &=& \sum_{\substack{\bm{0}\leq \bm{f},\bm{g}\\ \bm{f}+\bm{g}\leq \bm{q-1}}}
  c_{\bm{f}+\bm{g}} \binom{\bm{f}+\bm{g}}{\bm{f}} \left( \bigl({\pfrak_t^{(1)}}\bigr)^n {\omega } \right)^{\bm{g}} ({\omega^{(n)}})^{\bm{f}} {b_t}^{\,\bm{g}} \in A_{n}[\zeta]
\end{eqnarray}
By Lagrange interpolation \ref{lemma:lagrange-interpolation}, we therefore get
\[  \sum_{\bm{0}\leq \bm{f}\leq (\bm{q-1})-\bm{g}}
  c_{\bm{f}+\bm{g}} \binom{\bm{f}+\bm{g}}{\bm{f}} \left( \bigl({\pfrak_t^{(1)}}\bigr)^n {\omega } \right)^{\bm{g}} ({\omega^{(n)}})^{\bm{f}}\in A_{n}[\zeta] \]
  for all $\bm{g}$.
Let $\bm{g}$ be such that
$c_{\bm{f}+\bm{g}}\in A_{n-1}[\zeta]$ for all $\bm{f}>\bm{0}$, then we have
\[  c_{\bm{g}} \cdot \left( \bigl({\pfrak_t^{(1)}}\bigr)^n {\omega } \right)^{\bm{g}} \in A_{n}[\zeta].
\]
Hence, $c_{\bm g}$ can only have poles at the primes above $\pfrak$. In Proposition \ref{prop:basis-for-extension}, however, we already showed that the $c_{{\bm e}}$ do not have poles there. So $c_{\bm{g}}$ has to be integral. Inductively this shows that all the coefficients are integral.

The final claim \ref{item:integral-basis} of the statement follows by induction from claims \ref{item:n=0} and \ref{item:n>0}.
\end{proof}

\subsection{Field normal basis}

The degree of the extension $K_0(\zeta)/K(\zeta)$ is prime to the characteristic, and we have an integral basis  $\left\{ {\omega}^{\bm{e}} : \bm{0}\leq \bm{e}< \bm{q-1}\right\}$ with each ${\omega}^{\bm{e}}$ being a basis vector for the isotypic component of a character of the Galois group. Hence, a suitable linear combination of those leads to an integral normal basis of $A_0[\zeta]/A[\zeta]$.

For the extension $K_n(\zeta)/K_0(\zeta)$ the situation is totally different. Here, the degree of the extension is a power of the characteristic and the primes $(\theta-\zeta_i)$ are totally wildly ramified. Hence, there is no integral normal basis of the extension $K_n(\zeta)/K_0(\zeta)$.
However, by the normal basis theorem every finite field extension has a normal basis, and we show in the following that the Galois orbit of our top digit derivative basis element
$\prod_{k=1}^n ({\omega^{(k)}})^{\bm{q-1}}$ is such a field normal basis.

\begin{prop}[{\bf A field normal basis}] \label{prop:normal-basis}
Let
\[\eta_n:=\prod_{k=1}^n ({\omega^{(k)}})^{\bm{q-1}}\in K_{n}(\zeta).\]
Then the Galois orbit of $\eta_n$  is a normal basis for the extension $K_{n}(\zeta)$ over $K_0(\zeta)$.
\end{prop}

\begin{proof}
Allover we do an induction on $n$.\\
For $n=0$, this means $\eta_0=1$ is a normal basis for $K_0(\zeta)$ over itself which is trivially fulfilled.
Hence, as induction hypothesis we assume that $\eta_{n-1}$ generates a normal basis for $K_{n-1}(\zeta)$ over $K_{0}(\zeta)$.

As we already know that all the elements $\prod_{k=1}^n ({\omega^{(k)}})^{\bm{e_k}}$ with $\bm{0}\leq \bm{e_k}\leq \bm{q-1}$ are a basis of the extension $K_{n}(\zeta)$ over $K_{0}(\zeta)$, we are going to show that they all lie in the $K_{0}(\zeta)$-span of the Galois-conjugates of $\eta_n=\eta_{n-1}\cdot ({\omega^{(n)}})^{\bm{q-1}}$.

\textbf{First step:} All $\eta_{n-1}\cdot ({\omega^{(n)}})^{\bm{e_n}}$ ($\bm{0}\leq \bm{e_n}\leq \bm{q-1}$) lie in the $K_{0}(\zeta)$-span of the Galois-conjugates of $\eta_n$.

Applying $1+\pfrak^{n}b\in (A/\pfrak^{n+1}A)^\times$ as in the previous proof, we obtain
\begin{eqnarray*}
\sigma_{1+\pfrak^{n}b}\left( \eta_{n-1}\cdot ({\omega^{(n)}})^{\bm{q-1}} \right)
&=& \eta_{n-1}\cdot \left( \sum_{\bm{f}+\bm{e}= \bm{q-1}} \binom{\bm{q-1}}{\bm{e}} \left( \bigl({\pfrak_t^{(1)}}\bigr)^n {b_t  \omega } \right)^{\bm{f}} ({\omega^{(n)}})^{\bm{e}} \right) \\
&=&   \sum_{\bm{f}+\bm{e}= \bm{q-1}} (-1)^{|\bm{e}|} \eta_{n-1}\left( \bigl({\pfrak_t^{(1)}}\bigr)^n {\omega } \right)^{\bm{f}}({\omega^{(n)}})^{\bm{e}} {b_t}^{\bm{f}}.
\end{eqnarray*}
By Lagrange interpolation \ref{lemma:lagrange-interpolation}, there are $\FF_q(\zeta)$-linear combinations of these Galois conjugates being equal to
$ \eta_{n-1}\left( \bigl({\pfrak_t^{(1)}}\bigr)^n {\omega } \right)^{\bm{f}}({\omega^{(n)}})^{\bm{e}}$, and hence $K_{0}(\zeta)$-linear combinations being equal to
$$ \eta_{n-1}({\omega^{(n)}})^{\bm{e}}.$$

\textbf{Second step:} By induction on $|\bm{e_n}|$ we show that all  $\prod_{k=1}^n ({\omega^{(k)}})^{\bm{e_k}}$ lie in the $K_{0}(\zeta)$-span of the Galois-conjugates of $\eta_n$.

The case $|\bm{e_n}|=0$ (i.e. $\bm{e_n}=\bm{0}$) is a consequence of the first step and the assumption the $\eta_{n-1}$ generates a normal basis for $K_{n-1}(\zeta)$ over $K_{0}(\zeta)$.\\
Let $V$ denote the $K_{0}(\zeta)$-space generated by all $\prod_{k=1}^{n-1} ({\omega^{(k)}})^{\bm{e_k}}\cdot ({\omega^{(n)}})^{\bm{f}}$ with $|\bm{f}|<|\bm{e_n}|$.
As induction hypothesis, we assume that $V$ lies inside the $K_{0}(\zeta)$-span of the Galois-conjugates of $\eta_n$.\\
By the first step, $\eta_{n-1}\cdot ({\omega^{(n)}})^{\bm{e_n}}$ is in the Galois-span, and for all $a\in (1+\pfrak A)/(1+\pfrak^{n+1}A)\subset (A/\pfrak^{n+1} A)^\times$, we have
\begin{eqnarray*}
\sigma_a(\eta_{n-1}\cdot ({\omega^{(n)}})^{\bm{e_n}} ) &=& \sigma_a(\eta_{n-1})\cdot \left( {\omega^{(n)}}+{a_t^{(1)}\omega^{(n-1)}}+\dots + {a_t^{(n)}\omega} \right)^{\bm{e_n}} \\
&\equiv &  \sigma_a(\eta_{n-1})\cdot {{\omega^{(n)}}}^{\bm{e_n}} \mod{V}
\end{eqnarray*}
As $\eta_{n-1}$ generates a normal basis, there is some $K_{0}(\zeta)$-linear combinations of the
$\sigma_a(\eta_{n-1})\cdot ({\omega^{(n)}})^{\bm{e_n}} $ congruent to
$\prod_{k=1}^n ({\omega^{(k)}})^{\bm{e_k}}$ modulo $V$. Hence, $\prod_{k=1}^n ({\omega^{(k)}})^{\bm{e_k}}$ lies in the $K_{0}(\zeta)$-span of the Galois-conjugates of $\eta_n$.
\end{proof}

\begin{rem}
Though there can be no integral normal basis of the extensions $A_n[\zeta] / A[\zeta]$, when $n \geq 1$, due to wild ramification, one may ask a more refined question, following Leopoldt: Can the extension $A_n[\zeta] / A[\zeta]$ be cyclically generated by the maximal order $\mathcal{R}_n \subset K(\zeta)[G_n]$ such that $\mathcal{R}_n A_n[\zeta] \subset A_n[\zeta]$, where $G_n = \Gal(K_n(\zeta) / K(\zeta))$? A. Aiba has given some counter examples to such questions in this context; see \cite{Aiba1,Aiba2}.
\end{rem}


\begin{thebibliography}{99}

\bibitem{Aiba1} A. Aiba: \emph{Carlitz Modules and Galois Module Structure.} J. Number  Theory \textbf{62} (1997) pp.  213--219.

\bibitem{Aiba2} A. Aiba: \emph{Carlitz Modules and Galois Module Structure II.} J. Number  Theory \textbf{68} (1998) pp.  29--35.

\bibitem{AGtmot} G. W. Anderson: {\it $t$-motives.} Duke Math. J. {\bf 53} (1986) pp. 457--502.

\bibitem{AGTDann} G. W. Anderson and D. S. Thakur: {\it Tensor powers of the Carlitz module and zeta values.} Ann. of Math. (2) {\bf 132} (1990) pp. 159--191.

\bibitem{BAjtnb} B. Angl\`es: {\it Bases normales relatives en caract\'eristique positive.}, J. Th\'eor. Nombres Bordeaux {\bf 14} (2002) pp. 1--17.

\bibitem{APinv} B. Angl\`es and F. Pellarin: {\it Universal Gauss-Thakur sums and $L$-series.} Invent. Math. {\bf 200} (2014) pp. 653--669.

\bibitem{ATplms} B. Angl\`es and L. Taelman: {\it Arithmetic of characteristic p special L-values.} Proc. London Math. Soc. {\bf 110} (2015) pp. 1000--1032.

\bibitem{CKdp} K. Conrad: {\it The Digit Principle.} J. Number Theory {\bf 84} (2000) pp. 230--257.

\bibitem{FPannals} F. Pellarin: \emph{Values of certain $L$-series in positive characteristic}. Ann. of Math. (2) {\bf 176} (2012) pp. 1--39.

\bibitem{FPreps} F. Pellarin: \emph{A note on certain representations in characteristic $p$ and associated functions.} arXiv:1603.07912v2

\bibitem{APTR} F. Pellarin, B. Angl\`es, and F. Tavares Ribeiro: {\it Arithmetic of positive characteristic $L$-series values in Tate algebras.} Comp. Math. {\bf 152} (2016) pp. 1--61.

\bibitem{Rosen} M. Rosen: {\it Number Theory in Function Fields.} (2002) Springer-Verlag, New York.

\end{thebibliography}
\end{document}